\documentclass[reqno, oneside, 11pt]{amsart}
\usepackage[letterpaper]{geometry}
\usepackage{appendix}
\usepackage{amsmath}
\usepackage{amssymb}
\usepackage{enumerate,enumitem}
\usepackage{verbatim}
\usepackage{mathrsfs}
\usepackage{mathtools}
\usepackage{hyperref}
\usepackage{amsthm, mathdots, multicol,placeins}
\usepackage[dvipsnames]{xcolor}
\usepackage{ dsfont }
\usepackage{tikz}

\newcommand{\Z}{\mathbb{Z}}
\newcommand{\Q}{\mathbb{Q}}
\newcommand{\C}{\mathbb{C}}

\DeclareMathOperator{\Jac}{Jac}
\DeclareMathOperator{\USp}{USp}
\DeclareMathOperator{\SU}{SU}
\DeclareMathOperator{\Sp}{Sp}

\DeclareMathOperator{\U}{U}

\DeclareMathOperator{\SL}{SL}

\DeclareMathOperator{\ST}{ST}
\DeclareMathOperator{\Gal}{Gal}

\DeclareMathOperator{\End}{End}
\DeclareMathOperator{\Aut}{Aut}
\DeclareMathOperator{\diag}{diag}

\DeclareMathOperator{\TL}{TL}

\DeclareMathOperator{\AST}{AST}
\DeclareMathOperator{\Hg}{Hg}

\DeclareMathOperator{\LL}{L}
\DeclareMathOperator{\tr}{tr}

\DeclareMathOperator{\MT}{MT}

\newcommand{\matrixtwo}[4]{\begin{bsmallmatrix}#1&#2\\#3&#4\end{bsmallmatrix}}
\newcommand{\greyzero}{\textcolor{gray}{0}}

\newtheorem{theorem}{Theorem}[section]
\newtheorem{example}[theorem]{Example}
\newtheorem{proposition}[theorem]{Proposition}
\newtheorem{lemma}[theorem]{Lemma}
\newtheorem{definition}[theorem]{Definition}
\newtheorem{corollary}[theorem]{Corollary}

\newtheorem{remark}[theorem]{Remark}

\author{Justin Chen}
\address{Department of Mathematics, Brooklyn College, City University of New York; 2900 Bedford Avenue, Brooklyn, NY 11210 USA}
\email{justin.chen39@brooklyn.cuny.edu}

\author{Heidi Goodson}
\address{Department of Mathematics, Brooklyn College, City University of New York; 2900 Bedford Avenue, Brooklyn, NY 11210 USA}
\email{heidi.goodson@brooklyn.cuny.edu}

\author{Rezwan Hoque}
\address{Department of Mathematics, Brooklyn College, City University of New York; 2900 Bedford Avenue, Brooklyn, NY 11210 USA}
\email{rezwan.hoque98@bcmail.cuny.edu}

\author{Sabeeha Malikah}
\address{Department of Mathematics, Brooklyn College, City University of New York; 2900 Bedford Avenue, Brooklyn, NY 11210 USA}
\email{sabeeha.malikah18@bcmail.cuny.edu}

\title{Degeneracy and Sato-Tate groups of $y^2=x^{p^2}-1$}

\begin{document}

\begin{abstract}
    We say that an abelian variety is \emph{degenerate} if its Hodge ring is not generated by divisor classes. Degeneracy leads to some interesting challenges when computing Sato-Tate groups, and there are currently few examples and techniques presented in the literature. In this paper we focus on the Jacobians of the family of curves $C_{p^2}: y^2=x^{p^2}-1$, where $p$ is an odd prime. Using a construction developed by Shioda in the 1980s, we are able to characterize so-called \emph{indecomposable Hodge classes} as well as the Sato-Tate groups of these Jacobian varieties. Our work is inspired by computation, and examples and methods are described throughout the paper.
\end{abstract}

\maketitle

\section{Introduction}

Let $C_{p^2}/\Q$ be the hyperelliptic curve with affine model $y^2=x^{p^2}-1$, where $p$ is an odd prime. One of the main goals of this paper to is to describe the limiting distributions of coefficients of the normalized L-polynomials of curves in this family. These curves have complex multiplication by $\Q(\zeta_{p^2})$ and so, by results of Johansson in \cite{Johansson2017}, we know that the generalized Sato-Tate conjecture holds. Thus, the limiting distributions can be described through a study of the Sato-Tate groups of their Jacobians $\Jac(C_{p^2})$. One challenge to this is that these Jacobian varieties have degenerate Hodge rings by \cite[Theorem 1.1]{GoodsonDegeneracy2024}, and so the standard techniques used in this area of research must be modified. This makes this family of curves particularly interesting in the study of Sato-Tate groups.

Recall that an abelian variety is \emph{degenerate} if its Hodge ring is not generated by divisor classes. The Hodge classes not coming from divisor classes are called \emph{exceptional} and those not coming from Hodge classes in any lower codimension are called \emph{indecomposable}. Degeneracy in the Hodge ring affects several groups associated to an abelian variety, including the Hodge, Mumford-Tate, and Sato-Tate groups (see \cite{GoodsonDegeneracy2024}).

There are currently very few examples of Sato-Tate groups of degenerate abelian varieties in the literature. It was shown in \cite[Theorem 1.1]{GoodsonDegeneracy2024} that the Jacobians of curves of the form $C_m:y^2=x^m-1$ have degenerate Hodge rings when $m$ is an odd composite integer. That paper also explored how degeneracy can affect the Sato-Tate group of $\Jac(C_m)$ by studying the behavior for some small values of $m$. A method for computing Sato-Tate groups was developed in \cite{GalleseGoodsonLombardo2} that can be applied to the Jacobian of any curve in the family $y^2=x^m+1$. In our paper, we specialize to the family $C_{p^2}$ and are able to give an explicit description of the Sato-Tate groups of their Jacobians.

Results of \cite{GalleseGoodsonLombardo1} tell us that the degeneracy of $\Jac(C_{p^2})$ has a somewhat minimal effect on the Sato-Tate group. The indecomposable classes in the Hodge ring affect the Hodge group, which in turn affects the identity component of the Sato-Tate group. However, the component group is in some sense not affected. This is because the connected monomdromy field of $\Jac(C_{p^2})$ is exactly the endomorphism field, i.e. the field of definition of the endomorphisms, of $\Jac(C_{p^2})$. Thus, describing the indecomposable Hodge classes will lead us to a description of the Sato-Tate groups.

In this paper, we use a construction developed by Shioda in \cite{Shioda82} to give a complete description of indecomposable classes in the Hodge ring for the Jacobians of curves in our family $C_{p^2}$. This allows us to give the first explicit description in the literature of Sato-Tate groups of an infinite family of degenerate abelian varieties. Our work is inspired by computation: we implemented results of \cite{Shioda82} in order to study the indecomposable Hodge classes for the curve $C_m$ for odd composite $m$. It was only after computing scores of examples that we discovered that indecomposable classes for $m=p^2$ had a particularly nice and easy-to-describe format  (all of our code and data will be made available upon request). With further computation and study, one might be able to give an explicit description of these classes for more general values of $m$.

\subsection*{Organization of the paper}

In Section \ref{sec:background} we provide some important background information on degeneracy and Sato-Tate groups. We also discuss properties of the curves $C_{p^2}$ and their Jacobians. In Section \ref{sec:degeneracy} we introduce Shioda's construction for studying the Hodge ring of $\Jac(C_{m})$, where $m\geq 3$ is odd. We then give a series of results regarding Shioda's construction and degeneracy in the specific case where $m=p^2$. One of the most important results of this section is Corollary \ref{cor:indecomposableHodgeclasses} where we give a complete description of the indecomposable Hodge classes of $\Jac(C_{p^2})$.

In Section \ref{sec:SatoTateGroup} we apply the results of Section \ref{sec:degeneracy} to determine the identity component and the component group of the Sato-Tate group of $\Jac(C_{p^2})$ (see Theorem \ref{thm:STgroup}). Finally, in Section \ref{sec:momentstatistics} we describe our process for computing moment statistics. We work through an example to show how the degeneracy is taken into account in the computations.

\subsection*{Notation and conventions}

Throughout, let $p$ be an odd prime and $C_{p^2}/\Q$ denote the hyperelliptic curve $y^2=x^{p^2}-1$. We let $\Q(\zeta_{p^2})$ denote the $p^2$-cylcotomic field, where $\zeta_{p^2}:=\exp(2\pi i/p^2)$. For any rational number $x$ whose denominator is coprime to a positive integer $r$, $\langle x \rangle_r$ denotes the unique representative of $x$ modulo $r$ between 0 and $r-1$. 

Let $I$ denote the $2\times 2$ identity matrix and define the matrix
\begin{align*}
    J=\begin{pmatrix}0&1\\-1&0\end{pmatrix}.
\end{align*}

We embed $\U(1)$ in $\SU(2)$ via $u\mapsto U= \diag(u, \overline u)$. 
For any positive integer $n$, define the following subgroup of the unitary symplectic group $\USp(2n)$:
\begin{equation*}
    \U(1)^n:=\left\langle \diag( U_1, U_2,\ldots, U_n)\;|\; U_i\in \U(1)\right\rangle.
\end{equation*}

For an abelian variety $A$ defined over a number field $K$, we denote the Sato-Tate group of $A$ by $\ST(A) := \ST(A_K)$ with identity component denoted $ST^0(A)$ and component group $\ST(A)/\ST^0(A)$.

\subsection*{Acknowledgments}
Our research was partly supported by a grant from the National Science Foundation (DMS - 2201085). H.G. and S.M. were also supported through the Tow Mentoring and Research Program at Brooklyn College, City University of New York. We are grateful to the Tow Foundation for making this collaboration possible. 

We would like to thank Drew Sutherland for assisting us with the computation of the numerical moment statistics and the creation of the histogram in Section \ref{sec:momentstatistics}.

\section{Background}\label{sec:background}

\subsection{Background on degeneracy}\label{subsec:degeneracybackground}

In  this section we briefly recall some background information on Hodge rings and degeneracy (see \cite{GalleseGoodsonLombardo1,GalleseGoodsonLombardo2, GoodsonDegeneracy2024, Shioda82} for more details). For an abelian variety $A/\C$, we define its Hodge ring to be
$$\mathscr B^*(A):=\displaystyle\sum_{d=0}^{\dim(A)} \mathscr B^d(A), $$
where $\mathscr B^d(A)=(H^{2d}(A,\mathbb Q)\cap H^{d,d}(A))\otimes \mathbb C$ is the $\mathbb C$-span of Hodge classes of codimension $d$ on $A$. The Hodge classes of codimension $d=1$, i.e. $\mathscr B^1(A)$, are called the divisor classes of $A$. We say that an abelian variety $A$ is \emph{nondegenerate} if its Hodge ring is  generated by divisor classes and \emph{stably nondegenearate} if this is true for $A^n$ for all $n\in \Z_{\geq 1}$. Otherwise, we say that $A$ is \emph{degenerate}. The Hodge classes not generated by divisor classes are called \emph{exceptional}. In our work, we are particularly interested in \emph{indecomposable} Hodge classes, which are exceptional classes not generated by Hodge classes in any lower codimension. In other words, indecomposable classes in codimension $d$ are those that are not generated by $\mathscr B^1(A)+\mathscr B^2(A)+\cdots+ \mathscr B^{d-1}(A)$.

\subsection{Computation of the Sato-Tate group}\label{subsec:STgroupbackground}

The ease with which we can compute the Sato-Tate group of an abelian variety depends partly on the degeneracy of its Hodge ring. When the abelian variety is nondegenerate, we have a clear recipe for computing the Sato-Tate group as described below (see \cite{FGL2016,GoodsonCatalan, EmoryGoodson2022,EmoryGoodson2024, GoodsonHoque2024, LarioSomoza2018} for examples). 

Let $A$ be a dimension $g$ abelian variety defined over a number field $K$. The twisted Lefschetz group of $A$, denoted $\TL(A)$, is a closed algebraic subgroup of  $\Sp_{2g}$ defined by
\begin{align*}
\TL(A):=\bigcup_{\tau \in \Gal(\overline{K}/K)} \LL(A)(\tau),
\end{align*}
where $\LL(A)(\tau):=\{\gamma \in \Sp_{2g}\mid \gamma \alpha \gamma^{-1}=\tau(\alpha) \text{ for all }\alpha \in \End(A_{\overline{K}})\otimes \mathbb{Q}\}$. By Theorem 6.1 of \cite{Banaszak2015}, the algebraic Sato-Tate Conjecture holds for stably nondegenerate abelian varieties $A$ with $\AST(A)=\TL(A)$. We also note that, in this case, the Hodge group coincides with the Lefschetz group, which is the identity component of $\TL(A)$.

When $A$ is degenerate, the computation of the Sato-Tate group is more complicated. As we will see in Section \ref{sec:degeneracyresults}, exceptional classes in the Hodge ring lead to interesting relations among elements in the Hodge group, causing it to be strictly smaller than the Lefschetz group. Additionally, the component groups of the Sato-Tate group and twisted Lefschetz group may no longer coincide, as explained below.

Let $F$ be the endomorphism field of $A$, i.e. the field of definition of the endomorphisms of $A$. In the nondegenerate case, the canonical surjection
$$\ST(A)/\ST^0(A) \rightarrow \Gal(F/K),$$
is an isomorphism (see \cite[Remark 6.4]{Banaszak2015}). More generally, there is an isomorphism between $\ST(A)/\ST^0(A)$ and $\Gal(K(\varepsilon_A)/K)$, where $K(\varepsilon_A)$ is a finite extension of $K$ called the \emph{connected monodromy field} (see \cite[Section 1]{GalleseGoodsonLombardo2}). The endomorphism field will be contained in the connected monodromy field, and this containment may be strict. When the connected monodromy field is larger than the endomorphism field, one must use alternate methods to compute the component group of the Sato-Tate group since one needs a description of the generators of the extension $K(\varepsilon_A)$ over $K$  (see \cite[Section 2]{GalleseGoodsonLombardo2}). However when $K(\varepsilon_A)=F$, these more advanced methods are not necessary and it suffices to study  the twisted Lefschetz group.

\subsection{Preliminaries on the curve $C_{p^2}$ and its Jacobian}\label{subsec:CurveJacbackground}

In this section we collect several important facts about the curve $C_{p^2}$ and its Jacobian $\Jac(C_{p^2})$. 

The genus of $C_{p^2}$ is $g=\frac{p^2-1}{2}$. It is known that $C_{p^2}$ has CM by $\Q(\zeta_{p^2})$ and that its automorphism group is a cyclic group of order $4g+2$ (see, for example, \cite{Muller2017}). Let $\alpha$ denote the automorphism 
\begin{align}\label{eqn:curveautomorphism}
    \alpha(x,y):=(\zeta_{p^2}x,-y)
\end{align}
of the curve $C_{p^2}$. The order of $\alpha$ is $2p^2$, which equals $4g+2$. Thus, $\alpha$ is a generator of $\Aut(C_{p^2})$. 

Some key properties of the Jacobian $\Jac(C_{p^2})$ are collected in the following lemma.

\begin{lemma}\label{lemma:decomposition}
We have the following isogeny over $\Q$:    
$$\Jac(C_{p^2})\sim \Jac(C_p)\times X_{p^2},$$
where $C_p/\Q$ denotes the curve $y^2=x^p-1$ and $X_{p^2}$ is a simple abelian variety of dimension $\phi(p^2)/2$ with complex multiplication by $\Q(\zeta_{p^2})$.
\end{lemma}
\begin{proof} 
From \cite[Theorem 3.0.1]{GalleseGoodsonLombardo1} we have that 
$$\Jac(C_{p^2})\sim \prod_{\substack{d \mid p^2,\\ d \neq 1,2}} X_d,$$
where $X_d$ is an abelian variety of dimension $\phi(d)/2$.  Thus,
$\Jac(C_{p^2})\sim X_p \times X_{p^2}$
(see also \cite[Section 4]{GoodsonDegeneracy2024}). Part (3) of \cite[Theorem 3.0.1]{GalleseGoodsonLombardo1} tells us that both $X_p$ and $X_{p^2}$ are absolutely simple abelian varieties with CM by $\Q(\zeta_p)$ and $\Q(\zeta_{p^2})$, respectively. Furthermore, part (2) of \cite[Theorem 3.0.1]{GalleseGoodsonLombardo1} tells us that $X_p$ is isogenous to the Jacobian variety $\Jac(C_p).$
\end{proof}

\subsubsection{Endomorphisms of the Jacobian}
First, we note that, by \cite[Prop 3.5.1]{GalleseGoodsonLombardo1}, the endomorphism field of $\Jac(C_{p^2})$ is the cyclotomic field $\Q(\zeta_{p^2})$, and part (3) of \cite[Theorem 7.2.12]{GalleseGoodsonLombardo1} tells us that this is also its connected monodromy field. Thus, the discussion in Section \ref{subsec:STgroupbackground} tells us that we can study the component group of Sato-Tate group of the Jacobian via the twisted Lefschetz group. The work in this section builds towards this by computing endomorphisms of the Jacobian obtained from the automorphisms of the curve $C_{p^2}$.

Throughout the paper, we use the usual basis for the $\Q$-vector space of regular differential 1-forms of the curve $C_{p^2}$:
\begin{align}\label{eqn:basisdifferentials}
    \left\{\omega_j:=\frac{x^{j-1}dx}{y}\right\}_{j=1}^{j=g},
\end{align}
where $g$ is the genus of the curve. The pullback of the regular 1-form $\omega_j$ with respect to the curve automorphism $\alpha$ in Equation \eqref{eqn:curveautomorphism} is 
\begin{equation}\label{eqn:pullbackdifferential}
    \alpha^{*}\omega_{j} = \frac{(\zeta_{p^2}x)^{j-1}d(\zeta_{p^2}x)}{-y} = - \zeta_{p^2}^{j} \omega_{j}.
\end{equation}

By taking the symplectic basis of $H_1(\Jac(C_{p^2})_\C,\C)$ corresponding to the basis $\{\omega_j\}_{j=1}^{g}$ of regular 1-forms (with respect to the skew-symmetric matrix $\diag(J,J,\ldots, J)$), we can express the endomorphism $\alpha$ corresponding to the automorphism in Equation \eqref{eqn:curveautomorphism} as
\begin{equation}\label{eqn:alphaendomorphism}
\alpha=\diag(Z,Z^2,Z^3,\ldots,Z^{g}),
\end{equation}
where $Z=-\diag(\zeta_{p^2},\overline\zeta_{p^2})$.

\section{Degeneracy}\label{sec:degeneracy}

It was proven in \cite[Theorem 1.1]{GoodsonDegeneracy2024} that the Jacobian of the curve $C_{p^2}$ is degenerate, and it was shown that there are exceptional classes in codimension $d=\frac{p+1}{2}$ in \cite[Lemma 3.3]{GoodsonDegeneracy2024} (in fact, the classes are indecomposable). The focus of those results was proving the degeneracy of the Jacobian varieties, and finding one codimension containing exceptional classes was sufficient for this. In our current work, we go a step further and prove that, for $\Jac(C_{p^2})$ this is the only codimension containing indecomposable classes. We also give the exact number of indecomposable classes and an explicit description of them.

\subsection{General degeneracy results}\label{sec:degeneracyresults}

Recall from Lemma \ref{lemma:decomposition} that $\Jac(C_{p^2})\sim \Jac(C_p)\times X_{p^2},$
where $C_p/\Q$ denotes the curve $y^2=x^p-1$ and $X_{p^2}$ is a simple abelian variety of dimension $\phi(p^2)/2$. We begin this section with results for the factors of $\Jac(C_{p^2})$.

\begin{lemma}\label{lemma:nondegeneratefactors}
    The abelian varieties $\Jac(C_p)$ and $X_{p^2}$ are both nondegenerate abelian varieties. 
\end{lemma}
\begin{proof}
    This is the result in Theorem 7.2.3 of \cite{GalleseGoodsonLombardo1}, which states that every simple factor of $\Jac(C_{p^2})$ is nondegenerate.
\end{proof}

\begin{lemma}\label{lemma:MTisomorphism}
    We have the following isomorphism of Mumford-Tate groups
    $$\MT(\Jac(C_{p^2})) \simeq \MT(X_{p^2}).$$
\end{lemma}
\begin{proof}
    This follows from Proposition 7.2.4 of \cite{GalleseGoodsonLombardo1} since $X_{p^2}$ is the largest simple factor of $\Jac(C_{p^2})$ by Lemma \ref{lemma:decomposition}.
\end{proof}

The Mumford-Tate and Hodge groups are closely related to each other: the Hodge group of $A$ is the connected component of the identity of $\MT(A)\cap \SL_{V}$ (see \cite[Definition 4.2]{Banaszak2015} or \cite[Section 2.2]{GoodsonDegeneracy2024}). This immediately gives us the following corollary of Lemma \ref{lemma:MTisomorphism}.

\begin{corollary}
We have the following isomorphism of Hodge groups:
    $$\Hg(\Jac(C_{p^2})) \simeq \Hg(X_{p^2}).$$
\end{corollary}

The Hodge group of an abelian variety $A/F$, where $F$ is a number field, is also closely related to the Lefschetz group, denoted $\LL(A)$, which is the connected component of the identity in the centralizer of the endomorphism ring $\End(A_{\overline F})_\Q$ in $\Sp_{V}$. In general, the Hodge group of $A$ is a subgroup of the Lefschetz group. As another corollary to Lemma \ref{lemma:MTisomorphism} we have the following result.

\begin{corollary}
    $\Hg(X_{p^2})= \LL(X_{p^2}).$
\end{corollary}
\begin{proof}
    The proof of \cite[Lemma 3.5]{FGL2016} demonstrates that if an abelian variety is nondegenerate then its Hodge group is isomorphic to its Lefschetz group.
\end{proof}

We now give an explicit description of the Lefschetz group of $X_{p^2}$.
\begin{lemma}\label{lem:Lefschetzfactor}
Let $g'=\phi(p^2)/2$. Then
$$\LL(X_{p^2}) \otimes_{\mathbb{Q}} \mathbb{Q}_{\ell}=\{\diag( x_1,y_1,\dots,x_{g'},y_{g'}) \in \mathbb{Q}_{\ell}^{\times}~|~x_1y_1=\cdots=x_{g'}y_{g'}=1\}.$$
\end{lemma}
\begin{proof}
    This follows from the nondegeneracy result of Lemma \ref{lemma:nondegeneratefactors} and the same proof techniques as \cite[Lemma 3.5]{FGL2016} and \cite[Proposition 3.3]{EmoryGoodson2022}.
\end{proof}

\begin{corollary}\label{cor:HodgeGroup}
    $\Hg(\Jac(C_{p^2}))\simeq \U(1)^{g'}$, where $g'=\phi(p^2)/2$.
\end{corollary}

\begin{proof}
    The above results tell us that $\Hg(\Jac(C_{p^2}))\simeq \Hg(X_{p^2})\simeq  \LL(X_{p^2})\simeq \U(1)^{g'}.$
\end{proof}

Since $\Jac(C_{p^2})$ is of CM-type, it follows from Theorem 17.3.5 of \cite{BirkenhakeLange2004} that the Hodge group is commutative. The result of Corollary \ref{cor:HodgeGroup} is more specific: it tells us that $\Hg(\Jac(C_{p^2}))$ is \emph{smaller than expected} in the sense of \cite{GoodsonDegeneracy2024}. If $\Jac(C_{p^2})$ was nondegenerate, then the Hodge group would be isomorphic to $\U(1)^g$, where $g=({p^2-1})/{2}$ is the genus of the curve $C_{p^2}$. However, Corollary \ref{cor:HodgeGroup} tells us that $\Hg(\Jac(C_{p^2}))$ is isomorphic to $\U(1)^{g'}$. In other words, if we identify an element of the Hodge ring of $\Jac(C_{p^2})$ with a $2g\times 2g$ matrix $U\in \U(1)^g$, then there must be dependencies among some entries of the matrix. More precisely, since 
$$g-g'=\frac{p^2-1}{2}-\frac{p(p-1)}{2}=\frac{p-1}{2},$$
it must be the case that $p-1$ entries in such a matrix $U$ are dependent on the other entries of $U$.

In order to describe the dependencies, we will use a result  from \cite{BirkenhakeLange2004} that relates the $\Q$-span of Hodge classes (in the Hodge ring) with the Hodge group. Note that, in the notation of Shioda, $\mathscr B^d(A)=H^{2d}_\text{Hodge}(A)\otimes\C$ (see \cite[Section 2]{GoodsonDegeneracy2024} for further details on the notation). 

\begin{theorem}\cite[Theorem 17.3.3]{BirkenhakeLange2004}\label{thm:Hodge}
Let $A$ be an abelian variety of dimension $g$. For any $1\leq d \leq g$, denote by 
$$H^{2d}_\text{Hodge}(A):=H^{2d}(A,\Q)\cap H^{d,d}(A)$$
the $\Q$-vector space of Hodge cycles of codimension $d$ on $A$. Then
$$H^{2d}_\text{Hodge}(A)= H^{2d}(A,\Q)^{\Hg(A)}.$$
\end{theorem}

This result tells us that Hodge classes $\mathscr B^d(A)$ are invariant under the Hodge group. In our work we will use the indecomposable Hodge classes (which we will provide a description for in Corollary \ref{cor:indecomposableHodgeclasses}) to identify the extra relations in the Hodge group. See \cite[Chapter 17.3]{BirkenhakeLange2004} for a complete description of the action and \cite[Section 6]{GoodsonDegeneracy2024} for a worked example. In the next section we present a construction of Shioda \cite{Shioda82} that can be used to better understand  indecomposable Hodge classes.

\subsection{Shioda's construction}

In \cite[Section 1]{Shioda82}, Shioda defines a set of tuples that can act as an index set for Hodge classes of codimension $d$ for the Jacobian variety $\Jac(y^2=x^m-1)$, where $m\geq 3$ is an odd integer. We introduce the notation that is needed for our results, though we present most of it for general (odd) $m$ and not just $m=p^2$. Note that our notation differs slightly from that in \cite{Shioda82}.

\begin{definition}\label{def:theBset}
    Let $m$ be a positive, odd integer and $d$ be an integer satisfying $1\leq d \leq \frac{m-1}{2}$. 
    We define the set
    \begin{align}
        \mathfrak B_m^d:=\{\beta=(b_1,b_2,\ldots, b_{2d})\}
    \end{align}
    to be the set of tuples of length $2d$ satisfying the following properties:
    \begin{enumerate}
        \item $1\leq b_1<b_2<\cdots <b_{2d}\leq m-1$;\label{def:Bset_increasingentries}
        \item $\sum_{i=1}^{2d} b_i\equiv 0\pmod m $;\label{def:Bset_sumofentries} 
        \item $|t\cdot \beta|=d$ for all $t\in (\mathbb Z/m\mathbb Z)^\times$, where $|t\cdot\beta|=\sum_{i=1}^{2d} \langle tb_i\rangle_m/m$.\label{def:Bset_tbeta}
    \end{enumerate}
\end{definition}

The next two lemmas make the statements in Definition \ref{def:theBset} more precise. These lemmas also allow for more efficient computation as they significantly reduce the number of possible tuple combinations one needs to generate (see Remark \ref{rem:ShiodaTupleComputation}).

\begin{lemma}\label{lemma:betasumdm}
    Let $\beta=(b_1,\ldots,b_{2d})\in \mathfrak B_m^d$. Then $\sum_{i=1}^{2d} b_i=dm.$
\end{lemma}

\begin{proof}
    Let $\beta$ be as defined in the statement of the lemma. Property \ref{def:Bset_tbeta} of Definition \ref{def:theBset} tells us that $|t\cdot \beta|=d$ for all $t\in \mathbb Z/m\mathbb Z^*$. In particular, this will be satisfied for $t=1$, and so
       $$d=|1\cdot\beta|=|\beta|
         =\sum_{i=1}^{2d} \frac{\langle b_i\rangle_m}{m}
         =\sum_{i=1}^{2d} \frac{b_i}{m},$$
    where the last equality holds since $0<b_i<m$ for all $i$. Thus, $\sum_{i=1}^{2d} {b_i}=dm.$
\end{proof}

The next result gives us another important property for the entries in the tuples.

\begin{lemma}\label{lemma:betaentrybounds}
    Every $\beta=(b_1,\ldots,b_{2d})\in\mathfrak B_m^d$ satisfies $b_{d}< \frac{m}{2}$ and $b_{d+1}> \frac{m}{2}.$
\end{lemma}

\begin{proof}
    First note that, since $m$ is odd, it is relatively prime to 2. Thus, property \ref{def:Bset_tbeta} of Definition \ref{def:theBset} tells us that $|2\cdot \beta|=d$. 
    
    Suppose $b_{d+1}< \frac{m}{2}.$ Then 
    \begin{equation}\label{eqn:generalsumalpha2t}
        \sum_{i=1}^{2d} \langle 2b_i\rangle_m = \left(\sum_{i=1}^{d+1} 2b_i\right) + \left(\sum_{i=d+2}^{2d} \langle 2b_i \rangle_m \right).
    \end{equation} 
    Each $\langle 2b_i \rangle_m$, for $i> d$, equals $2b_i$ or $2b_i-m$. Thus,  
    $$\sum_{i=d+2}^{2d} \langle 2b_i \rangle\geq \sum_{i=d+2}^{2d} (2b_i-m)= \left(\sum_{i=d+2}^{2d} 2b_i\right) -(d-1)m.$$
    Combining this with Equation \eqref{eqn:generalsumalpha2t} yields
    \begin{align*}\sum_{i=1}^{2d} \langle 2b_i\rangle &\geq \left(\sum_{i=1}^{d+1} 2b_i\right)+\left(\sum_{i=d+2}^{2d} 2b_i\right) -(d-1)m\\
    &=\left(\sum_{i=1}^{2d} 2b_i\right)-(d-1)m\\
    &=2(dm)-(d-1)m\\
    &=m(d+1).
    \end{align*}
    But this would imply that $|2\cdot \beta|=d+1$, which is too large. Thus, we must have $b_{d+1}\geq \frac{m}{2},$ and so $b_{d+1}> \frac{m}{2}$ since $m$ is odd.

    Supposing $b_{d}> \frac{m}{2}$ and using similar techniques as above, one can show that we again get a contradiction. Thus, we must have $b_{d}< \frac{m}{2}$.
\end{proof}

\begin{remark}\label{rem:ShiodaTupleComputation}
To generate the tuple combinations for any given value of $m$ and $d$, we developed programs in both Python and Java. In the Python implementation, tuples were initially generated using the \texttt{combinations} function from the \texttt{itertools} module. These tuples were then further checked to ensure that all properties outlined in Definition \ref{def:theBset} were met. Incorporating the results from Lemmas \ref{lemma:betasumdm} and \ref{lemma:betaentrybounds} significantly improved the efficiency of the program, as they allowed for specific conditions to be added during tuple combinations. This reduced the total number of tuples being generated, in turn reducing the number of tuples that needed to be checked. See Remark \ref{rem:ShiodaTupleComputation2} for further improvements we made in the special case where $m=p^2$.
\end{remark}

We now define the terms \emph{exceptional} and \emph{indecomposable} in the context of the tuples in $\mathfrak B_m^d$.

\begin{definition} \label{def:exceptional} 
    We say that a tuple $\beta\in \mathfrak B_m^d$ is \textbf{exceptional} if it is not entirely made up of pairs $b_i, b_j$ such that $b_i+b_j\equiv 0 \pmod m$. 
\end{definition}

\begin{definition} \label{def:indecomposable} 
    We say that $\beta\in \mathfrak B_m^d$ is \textbf{indecomposable} if no proper subset (with an even number of elements) of $\{b_1,b_2,\ldots,b_{2d}\}$ adds to a multiple of $m$. Otherwise, we say that $\alpha$ is \textbf{decomposable}. An element is decomposable if it can be split into tuples coming from $\mathfrak B^{d_1}_{m},\ldots,\mathfrak B^{d_k}_{m}$, where $\sum_{i=1}^k d_i=d$ and $d_i<d$.
\end{definition}

Note that indecomposable tuples are also exceptional.

\begin{example}\label{ex:indecomposabletuples_m25}

For $m = 25$ and $d = 3$, there are four indecomposable tuples. These tuples are:
    \begin{itemize}
        \item (1, 6, 11, 16, 20, 21)
        \item (2, 7, 12, 15, 17, 22)
        \item (3, 8, 10, 13, 18, 23)
        \item (4, 5, 9, 14, 19, 24)
    \end{itemize}
We can easily verify that these tuples contain no proper subset of even size that sums to a multiple of $m$. In the case of $m = 25, d = 3$, there are 2971 tuples that satisfy properties \ref{def:Bset_increasingentries} and \ref{def:Bset_sumofentries} of Definition \ref{def:theBset}. After adding property \ref{def:Bset_tbeta}, that number is narrowed down to 224 tuples. From these 224 tuples, only 4 are exceptional tuples, all of which are indecomposable as well. 
\end{example}

A natural question to ask is: how many indecomposable tuples are there for a given pairing of $m$ and $d$? Shioda provides a lower bound for some values of $m$ and $d$ in \cite[Lemma 5.5]{Shioda82}. In our work we are focused on $m=p^2$, and the following corollary specifies Shioda's result to this case.

\begin{corollary}\label{cor:indecomptuples_lowerbound}
    Let $m=p^2$ and $d=\frac{p+1}{2}$. Then the number of indecomposable tuples in $\mathfrak B_m^d$, denoted $N_m(d)$, satisfies
$$N_m(d)\geq p-1.$$
\end{corollary}

In general, it seems difficult to predict the exact number of indecomposable tuples for a given $m$ and $d$. However, we will see below that the lower bound given in Lemma \ref{cor:indecomptuples_lowerbound} is the exact number of indecomposable tuples in $\mathfrak B_m^d$ for $m=p^2$ and $d=\frac{p+1}{2}$. Moreover there are no additional indecomposable tuples for $m=p^2$ and any other value of $d>1$.

Theorem 5.2 of \cite{Shioda82} describes the correspondence between the set $\mathfrak B_m^d$ in Definition \ref{def:theBset} and the Hodge classes of $\Jac(C_m)$. We present the part of the result that is relevant to our work below and encourage the reader to read the original text for the complete result. 

\begin{theorem}\cite[Theorem 5.2]{Shioda82}\label{thm:ShiodaHodgeClassTuple}
    Assume $m$ is odd. The Hodge classes on the Jacobian variety $\Jac(C_m)$ have the following description:
    $$\mathscr B^d(\Jac(C_m))=\bigoplus_{(b_1,\ldots,b_{2d})\in \mathfrak B_m^d} \mathbb C\, \omega_{b_1}\wedge \cdots \wedge \omega_{b_{2d}}. $$ 
\end{theorem}

The indecomposable Hodge classes are precisely those coming from indecomposable tuples in the correspondence in Theorem \ref{thm:ShiodaHodgeClassTuple}, and in the next section we use this correspondence to better understand the degeneracy of $\Jac(C_{p^2}).$

\subsection{Complete description of the indecomposable Hodge classes when $m=p^2$}

Theorem \ref{thm:ShiodaHodgeClassTuple} above gives the correspondence between Hodge classes and the tuples from Definition \ref{def:theBset}. We are particularly interested in characterizing the indecomposable Hodge classes of $\Jac(C_{p^2})$ since this will enable us to give an explicit description of the identity component of its Sato-Tate group. In this section, we give a complete description of the indecomposable tuples (Theorem \ref{thm:indecomposabletuples}) and the Hodge classes (Corollary \ref{cor:indecomposableHodgeclasses}).

In the proof of \cite[Lemma 5.5]{Shioda82}, Shioda defines a family of indecomposable tuples in order to give a lower bound on $N_m(d)$. We now specify this family to the case $m=p^2$ and $d=\frac{p+1}{2}$. For $1\leq i \leq p-1$, define
\begin{align}\label{eqn:ShiodaTuple}
    \beta_i:=(i, i+p, i+2p,\ldots, i+(p-1)p, p(p-i)).
\end{align}

Note that the values appearing in the tuple $\beta_i$ are all between 1 and $p^2-1$. However, the entries of the tuple $\beta_i$ may need to be permuted in order to be an element of the set $\mathfrak B_m^d$ (see property \ref{def:Bset_increasingentries} of Definition \ref{def:theBset}). In a slight abuse of notation, when we write $\beta_i$ we mean the permuted tuple.

\begin{remark}\label{rem:ShiodaTupleComputation2}
After analyzing the tuples, we observed that when $m$ is restricted to $p^2$, all indecomposable tuples (see Definition \ref{def:indecomposable}) come from Equation \eqref{eqn:ShiodaTuple}. We will prove that this is the case in Theorem \ref{thm:indecomposabletuples} below. This allowed us to create another program, specifically for the case of $m = p^2$, in which we implemented this equation using for loops.
\end{remark}

We now present three results that can be used to better understand the $\beta_i$ tuples. 

\begin{lemma}\label{lemma:p(p-i)position}
    Let $\beta_i$ be the tuple defined in Equation \eqref{eqn:ShiodaTuple}. When the entries in the tuple are written in increasing order then the index of the value $p(p-i)$ in $\beta_i$ is $p-i+1$. 
\end{lemma}
\begin{proof}
Let $1\leq i \leq p-1$. It is clear that $p(p-i)<i +(p-i)p$. Furthermore, we see that
$$i+(p-i-1)p=p(p-i)-(p-i)<p(p-i).$$
Thus, $i+(p-i-1)p < p(p-i) <i+(p-i)p$. This proves the result since the index of the value $i+(p-i-1)p$ in $\beta_i$ is $p-i$.
\end{proof}

\begin{lemma}\label{lemma:distinctbeta}
    Let $\beta_i$ and $\beta_j$ be two tuples of the form given in Equation \eqref{eqn:ShiodaTuple} with $i\not=j$. Then $\beta_i$ and $\beta_j$ have no entries in common.
\end{lemma}
\begin{proof}
Let $\beta_i$ and $\beta_j$ be as stated in the lemma. It is clear that $p(p-i)\not=p(p-j)$ when $i\not= j$. We will prove that the other entries appearing in the tuples are also distinct.

We first consider the entries $i+hp$ from $\beta_i$ and $j+kp$ from $\beta_j$, where $1\leq h,k\leq p-1$. If $h=k$, then it is clear that $i+hp\not= j+kp$ since $i\not= j$. Without loss of generality, we now assume that $h< k$ and write $k=h+a$ for some positive integer $a$. Then 
    $$j+kp=j+(h+a)p=j+ap + hp.$$
Note that $j+ap\not= i$ since $1\leq i \leq p-1$. Thus, $j+kp=j+ap + hp\not= i+hp.$

Finally, the entries $i+hp$ from $\beta_i$ and $p(p-j)$ from $\beta_j$ are distinct. Indeed, we easily see that $p\mid p(p-j)$ but $p\nmid(i+hp)$, and so the entries cannot be equal.

    Thus, $\beta_i$ and $\beta_j$ have no entries in common.
\end{proof}

\begin{corollary}
    All values in the range $[1,p^2-1]$ appear in exactly one tuple $\beta_i$.
\end{corollary}
\begin{proof}
    This follows from the fact that there are $p-1$ tuples of length $(p+1)$ of the form $\beta_i$ given in Equation \eqref{def:Bset_tbeta} and that every entry is distinct by Lemma \ref{lemma:distinctbeta}. Thus, there are $(p+1)(p-1)=p^2-1$ values appearing, all of which are positive and less than $p^2$.
\end{proof}

These results lead us to one of the main results of this section, which is a complete description of the indecomposable tuples for $m=p^2$.

\begin{theorem}\label{thm:indecomposabletuples}
    Let $m=p^2$ and $d=\frac{p+1}{2}$. The indecomposable tuples in $\mathfrak B_m^d$ are the $p-1$ tuples of the form $\beta_i$ given in Equation \eqref{eqn:ShiodaTuple}. Furthermore, there are no indecomposable in $\mathfrak B_m^d$ when $d\not=\frac{p+1}{2}$.
\end{theorem}
\begin{proof}   
   The result follows from the above lemmas and corollaries about the tuples, the correspondence between the tuples and Hodge classes from Theorem \ref{thm:ShiodaHodgeClassTuple}, and the degeneracy results of Section \ref{sec:degeneracyresults}. 

    Let $d=\frac{p+1}{2}$. Theorem \ref{thm:ShiodaHodgeClassTuple} gives the following correspondence between the Hodge classes and the tuples in the set $\mathfrak B_m^d$
    \begin{equation}\label{eqn:tupleHodgecorrespondence}
        (b_1,b_2,\ldots,b_{2d}) \longleftrightarrow \omega_{b_1}\wedge\omega_{b_2}\wedge \cdots \wedge \omega_{b_{2d}}.
    \end{equation}
   Let $\beta_i$ be a tuple of the form in Equation \eqref{eqn:ShiodaTuple}. Then $\beta_i$ can be identified with the Hodge class
   \begin{equation}\label{eqn:ShiodaHodgeClass}
       \nu_i=\omega_{i}\wedge\omega_{i+p}\wedge \omega_{i+2p}\wedge\cdots\wedge \omega_{i+(p-1)p}\wedge \omega_{p(p-i)}.
   \end{equation}
   
   The result in Theorem \ref{thm:Hodge} (i.e.~\cite[Theorem 17.3.3]{BirkenhakeLange2004}) gives the relationship between the Hodge group and the Hodge ring. The discussion after Corollary \ref{cor:HodgeGroup} tells us that we can identify elements of the Hodge group of $\Jac(C_{p^2})$ with $2g\times 2g$ matrices $U \in \U(1)^g$ of the form
   \begin{equation}\label{eqn:generalUmatrix}
       U=\diag(u_1,\overline u_1,u_2,\overline u_2, \ldots, u_g, \overline u_g),
   \end{equation}
   where $g$ is the genus of the curve $C_{p^2}$ and $u_j\overline u_j=1$ for $1\leq j\leq g$. By Theorem \ref{thm:Hodge}, we must have $U\cdot \nu_i=\nu_i$, where 
   \begin{equation}\label{eqn:Hodgegroupaction}
       U\cdot \nu_i = u_{i}u_{i+p}u_{i+2p}\cdots u_{i+(p-1)p}u_{p(p-i)}\nu_i.
   \end{equation}
    In other words, we must have $u_{i}u_{i+p}u_{i+2p}\cdots u_{i+(p-1)p}u_{p(p-i)}=1$. Thus, each tuple $\beta_i$ leads to a single relation among the entries of the matrix $U$. 
    
    As shown in Lemma \ref{lemma:distinctbeta}, the entries in tuples $\beta_i$ and $\beta_j$ are distinct whenever $i\not=j$, and so each of the tuples will lead to a different relation among the entries of the matrix $U$. Since there are $p-1$ tuples of the form $\beta_i$, this accounts for all of the extra relations in the matrix $U$ (see the discussion after Corollary \ref{cor:HodgeGroup}). Thus, there cannot be additional indecomposable tuples in $\mathfrak B_m^d$ or in $\mathfrak B_m^{d'}$ for any positive integer $d'$  since these would lead to additional relations in the matrix $U$.   
\end{proof}

\begin{corollary}\label{cor:indecomposableHodgeclasses}
    Let $d=\frac{p+1}{2}$. The indecomposable Hodge classes of codimension $d$ are given by
    $$\nu_i=\omega_{i}\wedge\omega_{i+p}\wedge \omega_{i+2p}\wedge\cdots\wedge \omega_{i+(p-1)p}\wedge \omega_{p(p-i)},$$
    where $1\leq i \leq p-1$.
\end{corollary}
\begin{proof}
    This follows from Theorem \ref{thm:indecomposabletuples} and the correspondence in Equation \eqref{eqn:tupleHodgecorrespondence}.
\end{proof}

\begin{remark}\label{rem:conjugationconvention}
Let $\beta_i$ be a tuple of the form in Equation \eqref{eqn:ShiodaTuple}. We will modify the elements of $\beta_i$ such that every entry $b_j$ with $j >d=\frac{p+1}{2}$ is written as $b_j - p^2$. This modification will negate elements of the tuple whose value is greater than $\frac{p^2}{2}$. This corresponds to expressing the differential $\omega_{b_j}$ as $\overline\omega_{p^2-b_j}$.
    
    After modifying the tuples in this way, we obtain pairs of tuples such that each $\beta_i$ generated from Equation \eqref{eqn:ShiodaTuple} is paired with the corresponding tuple $\beta_{p-i}$. For instance, for $p^2 = 9$, we have the tuples  $\beta_1 = (1, 4, -3, -2)$ and $\beta_2 = (2, 3, -4, -1)$. These correspond to the Hodge classes $\nu_1=\omega_1\wedge\omega_4\wedge\overline\omega_3\wedge\overline\omega_2$ and $\nu_2=\omega_2\wedge\omega_3\wedge\overline\omega_4\wedge\overline\omega_1$, respectively, and from either we can read off the effect on the Hodge group. For this reason, we will only focus on the tuples $\beta_i$ produced by Equation \eqref{eqn:ShiodaTuple} where $1\leq i \leq \frac{p-1}{2}$.
\end{remark}

\begin{lemma}\label{lemma:alwaysnegatep-i}
    Let $1\leq i \leq \frac{p-1}{2}$. When applying the conventions described in Remark \ref{rem:conjugationconvention} to the tuples $\beta_i$, the value $p(p-i)$ will always be negated. Furthermore, the Hodge class associated to the tuple $\beta_i$ will have $\overline{\omega}_{pi}$ as one of its factors.
\end{lemma}
\begin{proof}
    The first statement follows from Lemma \ref{lemma:p(p-i)position}, where we proved that the index of the value $p(p-i)$ in $\beta_i$ is $p-i+1$. When $1\leq i \leq \frac{p-1}{2}$, this value will always be in the second half of the tuple since $p-i+1> p-\frac{p-1}{2}=\frac{p+1}{2}$ and the tuple is of length $p+1$. The second statement follows since $p^2-p(p-i)=pi$. 
\end{proof}
\begin{corollary}\label{cor:tuplebeta_notnegated}
    Let $1\leq i \leq \frac{p-1}{2}$. When applying the conventions described in Remark \ref{rem:conjugationconvention} to the tuples $\beta_i$, the values from $i$ to $i+\frac{p-1}{2}p$ will never be negated.
\end{corollary}
\begin{proof}
    This follows from Remark \ref{rem:conjugationconvention} and Lemma \ref{lemma:alwaysnegatep-i} since these are the first half of the entries in $\beta_i$.
\end{proof}

\begin{corollary}\label{cor:firsthalfindecomposableHodgeclasses}
    Let $1\leq i \leq \frac{p-1}{2}$. Then 
    $$\nu_i=\omega_{i}\wedge\omega_{i+p}\wedge \omega_{i+2p}\wedge\cdots\wedge\omega_{i+p\frac{p-1}{2}}\wedge \overline\omega_{p\frac{p-1}{2}-i}\wedge\cdots\wedge \overline\omega_{p-i}\wedge \overline\omega_{pi}.$$
\end{corollary}
\begin{proof}
    This follows from Corollary \ref{cor:tuplebeta_notnegated} and the correspondence in Equation \eqref{eqn:tupleHodgecorrespondence}.
\end{proof}

\begin{example}\label{ex:Hodgeclasses_m25}
Let $m=25$. The tuples listed in Example \ref{ex:indecomposabletuples_m25} are the only indecomposable tuples for $m$ and they are all of the form $\beta_i$ in Equation \eqref{eqn:ShiodaTuple} with $1\leq i \leq 4$. Following the conventions described in Remark \ref{rem:conjugationconvention}, we select the first two tuples and adjust each entry of the tuple whose value is greater than $\frac{m}{2}$ to obtain
\begin{align*}
    \beta_1&=(1, 6, 11, -9, -5, -4)&
    \beta_2&=(2, 7, 12, -10, -8, -3).
\end{align*}
These correspond to the Hodge classes
\begin{align*}
    \nu_1&=\omega_1\wedge\omega_6\wedge\omega_{11}\wedge\overline\omega_9\wedge\overline\omega_5\wedge\overline\omega_4&
    \nu_2&=\omega_2\wedge\omega_7\wedge\omega_{12}\wedge\overline\omega_{10}\wedge\overline\omega_8\wedge\overline\omega_3.
\end{align*}
Let $U\in \U(1)^g$. From Equation \eqref{eqn:Hodgegroupaction} we obtain the following relations among entries of U:
\begin{align*}
    u_{11}&=\overline u_1 u_4 u_5\overline u_6 u_9&
    u_{12}&=\overline u_2 u_3 \overline u_7 u_8 u_{10}.
\end{align*}
\end{example}

\section{The Sato-Tate Group of $\Jac(C_{p^2})$}\label{sec:SatoTateGroup}

In this section we determine the identity component and the component group of the Sato-Tate group of $\Jac(C_{p^2})$. This Jacobian variety is degenerate by \cite[Theorem 1.1]{GoodsonDegeneracy2024}, and we gave a description of the indecomposable classes of the Hodge ring in Corollary \ref{cor:indecomposableHodgeclasses}. These indecomposable Hodge classes are used when computing the identity component of the Sato-Tate group.

First, recall from Section \ref{subsec:CurveJacbackground} that the dimension of $\Jac(C_{p^2})$ is $g=\frac{p^2-1}{2}$. From Lemma \ref{lemma:decomposition} we have that $\Jac(C_{p^2})\sim \Jac(C_p)\times X_{p^2},$ where $X_{p^2}$ is an absolutely simple abelian variety of dimension $g'=\frac{p(p-1)}{2}$. The first result of this section gives a description of the identity component of the Sato-Tate group of $\Jac(C_{p^2})$.

\begin{proposition}\label{prop:ST_identitycomponent}
    The identity component of the Sato-Tate group of $\Jac(C_{p^2})$ is isomorphic to $\U(1)^{g'}$. We can identify elements of the identity component with matrices $U=\diag(U_1,U_2,\ldots, U_g)$ in $\U(1)^g$ where 
    \begin{equation}\label{eqn:identitycomponentrelation}
        U_{{i+p\frac{p-1}{2}}}=\overline U_{i} \overline U_{i+p} \overline U_{i+2p}\cdots \overline U_{i+p\frac{p-3}{2}}  U_{p\frac{p-1}{2}-i} \cdots  U_{p-i}  U_{pi}
    \end{equation}
    for $1\leq i\leq \frac{p-1}{2}$.
\end{proposition}

\begin{proof}

Let $d=\frac{p+1}{2}$. Corollary \ref{cor:indecomposableHodgeclasses} gives a complete description of the indecomposable classes in this codimension and by Theorem \ref{thm:indecomposabletuples}, these are the only indecomposable classes in $\mathscr B(\Jac(C_{p^2}))$. Theorem \ref{thm:Hodge} (i.e., \cite[Theorem 17.3.3]{BirkenhakeLange2004}) describes the relationship between the Hodge group $\Hg(\Jac(C_{p^2}))$ and the Hodge ring $\mathscr B(\Jac(C_{p^2}))$, and this relationship has been made more explicit in Equation \eqref{eqn:Hodgegroupaction}.

Let $1\leq i\leq \frac{p-1}{2}$ and let $U$ be an element of the Hodge group $\Hg(\Jac(C_{p^2}))$. Consider the indecomposable classes given in Corollary \ref{cor:firsthalfindecomposableHodgeclasses}:
    $$\nu_i=\omega_{i}\wedge\omega_{i+p}\wedge \omega_{i+2p}\wedge\cdots\wedge\omega_{i+p\frac{p-1}{2}}\wedge \overline\omega_{p\frac{p-1}{2}-i}\wedge\cdots\wedge \overline\omega_{p-i}\wedge \overline\omega_{pi}.$$    
    From Equation \eqref{eqn:Hodgegroupaction} we have the action 
\begin{align*}
    U\cdot \nu_i &= U\cdot (\omega_{i}\wedge\omega_{i+p}\wedge \omega_{i+2p}\wedge\cdots\wedge\omega_{i+p\frac{p-1}{2}}\wedge \overline\omega_{p\frac{p-1}{2}-i}\wedge\cdots\wedge \overline\omega_{p-i}\wedge \overline\omega_{pi})\\
    &=(u_{i} u_{i+p} u_{i+2p}\cdots u_{i+p\frac{p-1}{2}}\overline u_{p\frac{p-1}{2}-i} \cdots\overline u_{p-i} \overline u_{pi})\nu_i.
\end{align*}

By Theorem \ref{thm:Hodge}, the Hodge group fixes elements in the Hodge ring and so we must have 
    $$u_{i} u_{i+p} u_{i+2p}\cdots u_{i+p\frac{p-1}{2}}\overline u_{p\frac{p-1}{2}-i} \cdots\overline u_{p-i} \overline u_{pi}=1.$$
One can easily show that the largest subscript appearing in the above equation is $i+p\frac{p-1}{2}$, and so we write 
    $$u_{i+p\frac{p-1}{2}}= \overline u_{i} \overline u_{i+p} \overline u_{i+2p}\cdots \overline u_{i+p\frac{p-3}{2}}  u_{p\frac{p-1}{2}-i} \cdots  u_{p-i}  u_{pi}.$$
This also yields the relation
    $$\overline u_{i+p\frac{p-1}{2}}= u_{i}u_{i+p} u_{i+2p}\cdots u_{i+p\frac{p-3}{2}}  \overline u_{p\frac{p-1}{2}-i} \cdots  \overline u_{p-i} \overline u_{pi}.$$

Thus, the matrix $U_{i+p\frac{p-1}{2}}$ must be as stated in Equation \eqref{eqn:identitycomponentrelation}. Furthermore, this accounts for all relations in the Hodge group since we have $\Hg(\Jac(C_{p^2}))\simeq \U(1)^{g'}$. This concludes the proof since the identity component of the Sato-Tate group is isomorphic to the Hodge group (see \cite[Lemma 2.8]{FKRS2012} and \cite[Remark 4.8]{Banaszak2015}, where the latter is relevant since the Mumford-Tate conjecture holds for CM abelian varieties by \cite{Pohlmann1968}).
\end{proof}

\begin{example}\label{ex:IDcomponent_m25}
 Let $p^2=25$. Proposition \ref{prop:ST_identitycomponent} tells us that the identity component $\ST^0(\Jac(C_{p^2}))$ is isomorphic to $\U(1)^{10}$. We can identify elements of the identity component with matrices of the form
 \begin{equation}\label{eqn:IDcomponent_m25}
     U=\diag(U_1, U_2, \ldots, U_{10},\overline U_{1}U_4U_5\overline U_{6}U_9,\overline U_{2}U_3\overline U_{7}U_8U_{10})
 \end{equation}
 in $\USp(24).$
\end{example}

We now turn to the component group of the Sato-Tate group. Though $\Jac(C_{p^2})$ is degenerate, we can use the usual techniques from the literature for computing the component group of the Sato-Tate group (see the discussion in Section \ref{subsec:STgroupbackground}). Our main goal at this point is to find a generator of the component group of the Sato-Tate group.

\begin{definition}\label{def:gammamatrix}
Let $p$ be an odd prime, $\tau_a$ be a generator of $\Gal(\Q(\zeta_{p^2})/\Q)$, and $S=\{1,\ldots, g\}$, where $g=\frac{p^2-1}{2}$ is the genus of the curve $C_{p^2}$. We define the matrix $\gamma$ in $\USp(2g)$ to be the $2g\times 2g$ block matrix whose $ij^{th}$ block is given by
\begin{equation*}
    \gamma[i,j] = \begin{cases}
        I & \text{if $j=\langle ai\rangle_{p^2}$,} \\ 
        J & \text{if $j=p^2-\langle ai\rangle_{p^2}$,} \\
        0 & \text{otherwise}.
    \end{cases}
\end{equation*}
\end{definition}

\begin{example}\label{ex:gamma_m25}
Let $p^2=25$. Using $\tau_2$ as a generator for the Galois group $\Gal(\Q(\zeta_{25})/\Q)$, Definition \ref{def:gammamatrix} yields the matrix
    \begin{equation}\label{eqn:gamma_m25}
    \gamma = {\small\left(\begin{array}{rrrrrrrrrrrr}
         \greyzero & I & \greyzero & \greyzero & \greyzero & \greyzero & \greyzero & \greyzero & \greyzero & \greyzero & \greyzero & \greyzero  \\
         \greyzero & \greyzero & \greyzero & I & \greyzero & \greyzero & \greyzero & \greyzero & \greyzero & \greyzero & \greyzero & \greyzero  \\
         \greyzero & \greyzero & \greyzero & \greyzero & \greyzero & I & \greyzero & \greyzero & \greyzero & \greyzero & \greyzero & \greyzero  \\
         \greyzero & \greyzero & \greyzero & \greyzero & \greyzero & \greyzero & \greyzero & I & \greyzero & \greyzero & \greyzero & \greyzero  \\
         \greyzero & \greyzero & \greyzero & \greyzero & \greyzero & \greyzero & \greyzero & \greyzero & \greyzero & I & \greyzero & \greyzero  \\
         \greyzero & \greyzero & \greyzero & \greyzero & \greyzero & \greyzero & \greyzero & \greyzero & \greyzero & \greyzero & \greyzero & I \\
         \greyzero & \greyzero & \greyzero & \greyzero & \greyzero & \greyzero & \greyzero & \greyzero & \greyzero & \greyzero & J & \greyzero  \\
         \greyzero & \greyzero & \greyzero & \greyzero & \greyzero & \greyzero & \greyzero & \greyzero & J & \greyzero & \greyzero & \greyzero  \\
         \greyzero & \greyzero & \greyzero & \greyzero & \greyzero & \greyzero & J & \greyzero & \greyzero & \greyzero & \greyzero & \greyzero  \\
         \greyzero & \greyzero & \greyzero & \greyzero & J & \greyzero & \greyzero & \greyzero & \greyzero & \greyzero & \greyzero & \greyzero  \\
         \greyzero & \greyzero & J & \greyzero & \greyzero & \greyzero & \greyzero & \greyzero & \greyzero & \greyzero & \greyzero & \greyzero  \\
         J & \greyzero & \greyzero & \greyzero & \greyzero & \greyzero & \greyzero & \greyzero & \greyzero & \greyzero & \greyzero & \greyzero  \\
    \end{array}\right)}.
\end{equation}
\end{example}

We will show that the group $\langle \gamma \rangle$ is isomorphic to the component group of the Sato-Tate group, but we first compute the inverse of the matrix $\gamma$.

\begin{lemma}\label{lemma:gammainverse}
    Let $\gamma$ be the matrix described in Definition \ref{def:gammamatrix}. Then $\gamma^{-1}=\gamma^{T}$.
\end{lemma}
\begin{proof}
    First note that $J^T=-J=J^{-1}$. The result then follows by computing the $ij^{th}$ block of $\gamma\cdot \gamma^T$ as follows
    $$\gamma\cdot \gamma^T[i,j] =\sum_{k=1}^g \gamma[i,k]\gamma^T[k,j]=\begin{cases}
        I^2 \text{ or } JJ^T & \text{if $i=j$,} \\
        0 & \text{if $i\not=j$}.\end{cases}$$
\end{proof}

\begin{theorem}\label{thm:STgroup}
    Let $g =\frac{p^2-1}{2}$ be the genus of the curve $C_{p^2}$ and let $g'=\frac{p(p-1)}{2}$. The Sato-Tate group of $\Jac(C_{p^2})$, up to isomorphism in $\USp(2g)$, is given by
    \begin{equation}\label{eqn:STgroup}
        \ST(\Jac(C_{p^2})) \simeq \langle \U(1)^{g'},\gamma \rangle,
    \end{equation}
    where the embedding of $\U(1)^{g'}$ in $\USp(2g)$ is described in Proposition \ref{prop:ST_identitycomponent}.
\end{theorem}

\begin{proof}

Since the component group of the Sato-Tate group is isomorphic to $\Gal(\Q(\zeta_{p^2})/\Q)$, which is cyclic, it suffices to prove that (1) $\gamma \in \ST(\Jac(C_{p^2}))$ and (2) its order is $\phi(p^2)$ in the component group $\ST(\Jac(C_{p^2}))/\ST^0(\Jac(C_{p^2}))$.

Let $\tau_a$ be a generator of $\Gal(\Q(\zeta_{p^2})/\Q)$. For (1), it suffices to show that $\gamma$ is an element of the twisted Lefschetz group of $\Jac(C_{p^2})$, i.e. that
$$\gamma \alpha \gamma^{-1}={}^{\tau_a} \alpha,$$
where $\alpha$ is as defined in Equation \eqref{eqn:alphaendomorphism}. This is easily done by noting that $J\matrixtwo{\zeta}{}{}{\overline \zeta}J^T=\matrixtwo{\overline \zeta}{}{}{\zeta}$ and then using the same arguments as in the proof of Theorem 3.4 of \cite{GoodsonHoque2024}.

Now let $n=\phi(p^2)$, which is the order of $\tau_a$ in $\Gal(\Q(\zeta_{p^2})/\Q)$, and let $r$ be the smallest positive integer such that $\gamma^r\in \ST^0(\Jac(C_{p^2}))$. Then $\gamma^r$ is a diagonal matrix, and each $2\times2$ block entry satisfies $\gamma^r[i,i]=\pm I$. Thus,
$${}^{\tau_a^r} \alpha=\gamma^r \alpha \gamma^{-r}=\alpha,$$
and so $r$ must be a multiple of the order of $\tau_a$. But it is also true that 
$$\gamma^n \alpha \gamma^{-n}={}^
{\tau_a^n} \alpha=\alpha,$$
and so $\gamma^n$ is an element of $\ST^0(\Jac(C_{p^2}))$. Hence, $r=\phi(p^2)$ and we have verified (2).

\end{proof}

\section{Moment Statistics}\label{sec:momentstatistics}

With the results of the previous sections, we can now describe the limiting distributions of coefficients of the normalized L-polynomials of the curves $C_{p^2}:y^2=x^{p^2}-1$. A common way to describe a distribution is through its moment statistics sequence. The generalized Sato-Tate conjecture, which is known to be true in the CM case \cite{Johansson2017}, predicts that the moment sequence for the distributions of coefficients of the normalized L-polynomials converges to the moment sequence of the Sato-Tate group. The moment statistics of the Sato-Tate group are often referred to as the \emph{theoretical moments} while those coming from the coefficients of the normalized L-polynomials are called the \emph{numerical moments}.

In what follows, we give some background information on moment statistics in the context of Sato-Tate groups (for more background we refer the reader to \cite[Section 5]{EmoryGoodson2024} and \cite[Sections 4.3, 4.4]{SutherlandNotes}). We then discuss the computation methods for a specific example to demonstrate how the degeneracy of the Sato-Tate groups of our curves is taken into account.

\subsection{Computing theoretical moments}\label{sec:computingmoments}

The theoretical moments are computed from the coefficients of the characteristic polynomial of random conjugacy classes in the Sato–Tate group. The isomorphism in Theorem \ref{thm:STgroup} tells us that we can compute the moment sequence by working with the matrix group given in Equation \eqref{eqn:STgroup}.

Recall that for the unitary group $\U(1)$, the trace map $\tr$ on a random element $U\in \U(1)$ is given by $z:=\tr(U)=u+\overline{u}=2\cos(\theta)$, where $u=e^{i\theta}$. Then $dz=-2\sin(\theta)d\theta$ and 
$$\mu_{\U(1)}= \frac1{2\pi} \frac{dz}{\sqrt{4-z^2}}=\frac1{2\pi} d\theta$$
gives a uniform measure of $\U(1)$ on the eigenangle $\theta\in[-\pi,\pi]$ (see \cite[Section 2]{SutherlandNotes}). The $n^{th}$ moment $M_n[\mu]$  is the expected value of $\phi_n:z\mapsto z^n$ with respect to $\mu$, computed as
$$M_n[\mu] = \int_{I} z^n\mu(z),$$
where $I=[-2,2]$.

We extend this to the group given in Equation \eqref{eqn:STgroup} as follows. Let $U$ be a random matrix in the identity component described in Proposition \ref{prop:ST_identitycomponent} and $\gamma$ be the matrix defined in Definition \ref{def:gammamatrix}. Let $g_i^k$ denote the coefficient of $T^i$ in the characteristic polynomial $P(T)$ of $U\gamma^k$, where $0\leq k<\phi(p^2)$ (the order of $\gamma$). The $n^{th}$ moment $M_n[\mu_i^k]$ is the expected value of $(g_i^k)^n$, and we can compute this by integrating against the Haar measure. We obtain moment statistics for the full Sato-Tate group of $\Jac(C_{p^2})$ by taking the average of the moments for $U\gamma^k$.

\subsection{Example: $p^2=25$}
\label{sec:tableshistograms}

In this section we specialize to the genus 12 curve $C_{25}: y^2=x^{25}-1$. We use the techniques described in Section \ref{sec:computingmoments} to compute the theoretical moments coming from the Sato-Tate group of $\Jac(C_{p^2})$. Let $U$ be a random matrix of the form in Equation \eqref{eqn:IDcomponent_m25} and let $\gamma$ be the matrix in Equation \eqref{eqn:gamma_m25}. We compute the characteristic polynomial of each matrix $U\gamma^k$, where $0\leq k< \phi(25)=20$. We find that $g_1^k$-coefficient (corresponding to the characteristic polynomial of $U\gamma^k$) equals 0 unless $k$ is a multiple of 4, and the $g_1^k$-coefficient with the most terms occurs when $k=0$. To compute the moment statistics for these values of $k$, we integrate the expected value of $(g_1^k)^n$ against the Haar measure. 

The most complicated of these is in the $k=0$ case, where $M_n[\mu_1^0]$ is equal to the value of the integral
{\small\begin{align*}
    \frac{2^n}{(2\pi)^{10}}\int_0^{2\pi}\cdots\int_0^{2\pi} (
\cos\left(\theta_1\right)+
&\cdots +\cos\left(\theta_{10}\right) \\
&+\cos\left(-\theta_1+\theta_4+\theta_5-\theta_6+\theta_9\right)+\cos\left(-\theta_2+\theta_3-\theta_7+\theta_8+\theta_{10}\right))^n \,d\theta_1\cdots d\theta_{10}.
\end{align*}}
{\flushleft In the above integral we can see the effect of the degeneracy of the Sato-Tate group: the last two terms in the integrand show the additional relations among elements of the identity component. }

To compute $M_n[\mu_1^k]$ for $k=4,8,12,16$, we integrate
{\small\begin{align*}
    \frac{(\pm 2)^n}{(2\pi)^{2}}\int_0^{2\pi}\int_0^{2\pi} (
\cos\left(\theta_5\right) +\cos\left(\theta_{10}\right))^n \,d\theta_5d\theta_{10},
\end{align*}}
{\flushleft where the numerator of the coefficient is $2^n$ when $k=4,12$ and $(-2)^n$ when $k=8,16$. We compute the moment statistics $M_n[\mu_1]$ of the full Sato-Tate group by averaging over the size of the group.} 

The moment statistics we obtained were computed in Python and are displayed in Table \ref{table:m25a1}. The Python code implemented the \texttt{nquad} function from the \texttt{integrate} sub-package of the \texttt{SciPy} library and the symbolic \texttt{integrate} function from Python's \texttt{SymPy} library.
We found that the numerical integration function proved to be efficient when dealing with simple integrals, while symbolic integration was more efficient for the complicated integral we obtained for $k=0$.

In theory, this process can be completed for each of the coefficients $g_i^k$ of the characteristic polynomial of $U\gamma^k$. However, in practice, the coefficients become extremely unwieldy as $i$ increases, particularly for $k=0$. For example, we found that the $g_2^0$ coefficient has 265 terms which can be grouped as 132 pairs of the form $z, \overline{z}$, where $z$ is a root of unity, and one constant term that equals 12. The $g_i^0$ coefficients for larger values of $i$ have even more terms. Thus, it was not feasible to compute moment statistics for these coefficients of the characteristic polynomial.

The numerical moments coming from the $a_1$-coefficient of the normalized L-polynomial were computed for primes $p<2^{25}$ using the results of \cite{Sutherland2020}.  

\begin{table}[h]
\begin{tabular}{|c|p{2cm}|p{2cm}|p{2cm}|p{2cm}|p{2cm}|}
\hline
 &$M_2$ & $M_4$ & $M_6$ & $M_8$ \\ 
\hline
$a_1$& 2.009 & 90.848 & 9452.007 & 1438061.241\\
\hline
$\mu_1$ & 2 & 90 & 9344 & 1419866\\
\hline
\end{tabular}
\caption{Table of $a_1$- and $\mu_1$-moments for $C_{25}: y^2=x^{25}-1$ (with $p<2^{25}$).}\label{table:m25a1}
\end{table}

\vspace{-.1in}
In Figure \ref{fig:histogram_m25} below we display a histogram of the distribution of the numerical $a_1$-coefficients. 

\begin{figure}[h]
        \centering
        {
\setlength{\fboxsep}{0pt}\fbox{\includegraphics[width=.55\textwidth]{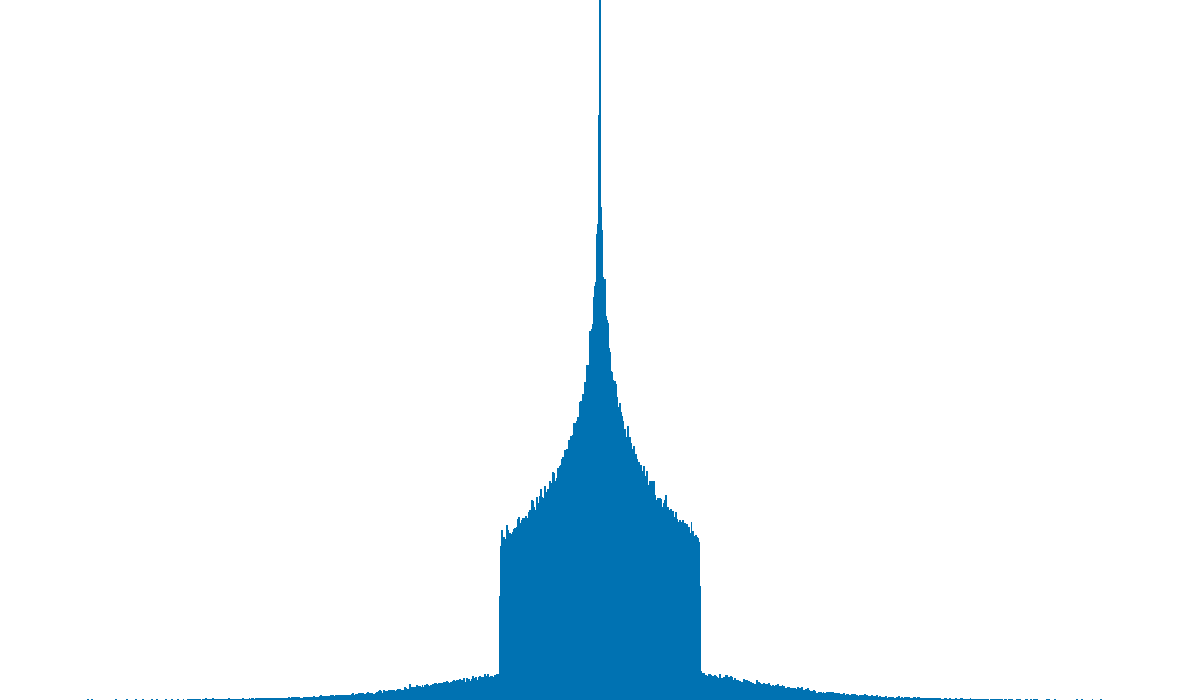}}}
        \caption{Histogram of the $a_1$-coefficients for $C_{25}: y^2=x^{25}-1$.}\label{fig:histogram_m25}
\end{figure}

\bibliographystyle{alpha}
\bibliography{biblio}

\end{document}